\documentclass[12pt]{amsart}

\usepackage[latin1]{inputenc}
\usepackage{amsmath,amsfonts,amssymb,mathrsfs,amsthm}
\usepackage{mathrsfs}
\usepackage{xpatch}
\usepackage{color}
\usepackage[pdftex]{hyperref}
\usepackage{amscd}

\newcommand{\C}{\mathscr{C}}

\newcommand{\R}{\mathbb{R}}
\newcommand{\N}{\mathbb{N}}

\newcommand{\K}{\mathbb{K}}
\newcommand{\eps}{\varepsilon}
\newcommand{\lraup}{\relbar\joinrel\rightharpoonup}
\newcommand{\bs}{\boldsymbol}

\numberwithin{equation}{section}

\newtheorem{thm}{Theorem}[section]

\newtheorem{rem}{Remark}[section]
\newtheorem{lemma}{Lemma}[section]
\newtheorem{prop}{Proposition}[section]

\newtheorem{example}{Example}[section]
\xpatchcmd{\proof}
{\itshape}
{\bfseries}
{}
{}

\begin{document}

\title[]{On nonlocal variational and quasi-variational inequalities with fractional gradient} 

\author[J.F. Rodrigues]{Jos\'e Francisco Rodrigues}
\address{CMAFcIO -- Departamento de Matemática, Faculdade de Ciências, Universidade de Lisboa
P-1749-016 Lisboa, Portugal}
\email{jfrodrigues@ciencias.ulisboa.pt}

\author[L. Santos]{Lisa Santos} 
\address{CMAT and Departamento de Matemática, Escola de Ciências, Universidade do Minho, Campus de Gualtar, 4710-057 Braga, Portugal}
\email{lisa@math.uminho.pt}

\begin{abstract}
We extend classical results on variational inequalities with convex sets with gradient constraint to a new class of fractional partial differential equations in a bounded domain with constraint on the distributional Riesz fractional gradient, the $\sigma$-gradient ($0<\sigma<1$). We establish continuous dependence results with respect to the data, including the threshold of the fractional $\sigma$-gradient. Using these properties we give new results on the existence to a class of quasi-variational variational inequalities with fractional gradient constraint via compactness and via contraction arguments. Using the approximation of the solutions with a family of quasilinear penalisation problems we show the existence of generalised Lagrange multipliers for the $\sigma$-gradient constrained problem, extending previous results for the classical gradient case, i.e., with  $\sigma=1$.
\end{abstract}

\maketitle

\section{Introduction}

\thispagestyle{empty}

\footnote{This is a revised and corrected version of the authors paper ``{\em On nonlocal variational and quasi-variational inequalities with fractional gradient}. Appl. Math. Optim. 80 (2019), no. 3, 835--852''.}In a series of two interesting papers \cite{ShiehSpector2015} and \cite{ShiehSpector2018}, Shieh and Spector have considered a new class of fractional partial differential equations. Instead of using the well-known fractional Laplacian, their starting concept is the distributional Riesz fractional gradient of order $\sigma\in(0,1)$, which will be called here the $\sigma$-gradient $D^\sigma$, for brevity: for $u\in L^p(\R^N)$, $1<p<\infty$, we set
\begin{equation}
\label{dsigma}
\big(D^\sigma u\big)_j=\frac{\partial^\sigma u}{\partial x_j^\sigma}=\frac{\partial\ }{\partial x_j}I_{1-\sigma}u,\qquad 0<\sigma<1,\quad j=1,\ldots,N,
\end{equation}
where $\frac{\partial\ }{\partial x_j}$ is taken in the distributional sense, for every $v\in\C^\infty_0(\R^N)$,
$$\Big\langle \frac{\partial^\sigma u}{\partial x_j^\sigma}, v\Big\rangle=-
\Big\langle I_{1-\sigma}u,\frac{\partial v }{\partial x_j} \Big\rangle=
-\int_{\R^N}(I_{1-\sigma}u)\frac{\partial v }{\partial x_j}\, dx,$$
with $I_\alpha$ denoting the Riesz potential of order $\alpha$, $0<\alpha<1$:
$$I_\alpha u(x)=(I_\alpha*u)(x)=\gamma_{N,\alpha}\int_{\R^N}\frac{u(y)}{|x-y|^{N-\alpha}}\,dy,\qquad\text{with }\gamma_{N,\alpha}=\frac{\Gamma(\frac{N-\alpha}2)}{\pi^\frac{N}2\,2^\alpha\,\Gamma(\frac\alpha2)}.$$

As it was shown in \cite{ShiehSpector2015}, $D^\sigma$ has nice properties for $u\in\C^\infty_0(\R^N)$, namely
\begin{equation}\label{propdsigma1}
D^\sigma u\equiv D(I_{1-\sigma}u)=I_{1-\sigma}*Du,
\end{equation}
\begin{equation}\label{propdsigma2}
(-\Delta)^\sigma u=-\sum_{j=1}^N\frac{\partial^\sigma\ }{\partial x_j^\sigma}\, \frac{\partial^\sigma\ }{\partial x_j^\sigma}\,u,
\end{equation}
where the well-known fractional Laplacian may be given, for a suitable constant $C_{N,\sigma}$, by (see, for instance, \cite{DiNezzaPalatucciValdinoci2012}):
$$(-\Delta)^\sigma u\equiv C_{N,\sigma}\text{ P.V.}\int_{\R^N}\frac{u(x)-u(y)}{|x-y|^{N+2\sigma}}\,dy.$$

It was also observed in \cite{ShiehSpector2018} that the $\sigma$-gradient is an example of the non-local gradients considered in \cite{MengeshaSpector2015}, which can be also given by
\begin{equation}\label{propdsigma3}
D^\sigma u(x)=R(-\Delta)^{\frac{\sigma}{2}} u(x)=(N+\sigma-1)\gamma_{N,1-\sigma}\int_{\R^N}\frac{u(x)-u(y)}{|x-y|^{N+\sigma}}\,\frac{x-y}{|x-y|}\,dy,
\end{equation}
in terms of the vector-valued Riesz transform (see \cite{Stein1970}, with $\rho_N=\text{\small $\Gamma\big(\frac{N+1}2\big)/\pi^{\frac{N+1}2}$}$):
$$Rf(x)=\rho_N\text{ P.V.}\int_{\R^N}f(y)\,\frac{x-y}{|x-y|^{N+1}}\,dy.$$

We observe that, from the properties of $D^\sigma$ and a result of \cite{Kurokawa1981} on the Riesz kernel as approximation of the identity as $\alpha\rightarrow0$, the $\sigma$-gradient approaches the standard gradient as $\sigma\rightarrow1$: if $Du\in L^p(\R^N)^N\cap  L^q(\R^N)^N$, $1<q<p$, then $D^\sigma u\underset{\sigma\rightarrow1}{\longrightarrow}Du$ in $L^p(\R^N)^N$.

Introducing the vector space of fractional differentiable functions as the closure of $\C^\infty_0(\R^N)$ with respect to the norm 
$$\|u\|_{\sigma,p}^p=\|u\|_{L^p(\R^N)}^p+\|D^\sigma u\|^p_{L^p(R^N)^N},\quad 0<\sigma<1,\ p>1,$$
by \cite[Theorem 1.7]{ShiehSpector2015} it is exactly the Bessel potencial space $L^{\sigma,p}(\R^N)\hookrightarrow W^{s,p}(\R^N)$, $0\le s<\sigma$, where $W^{s,p}(\R^N)$ denotes the usual fractional Sobolev space. In \cite{ShiehSpector2015} the solvability of the fractional partial differential equations with variable coefficients and Dirichlet data was treated in the case $p=2$, as well as the minimization of the integral functionals of the $\sigma$-gradient with $p$-growth was also considered, leading to the solvability of a fractional $p$-Laplace equation of a novel type.

In this work we are concerned with the Hilbertian case $p=2$ in a bounded domain $\Omega\subset\R^N$, with Lipschitz boundary, where the homogeneous Dirichlet problem for a general linear PDE with measurable coefficients is considered under an additional constraint on the $\sigma$-gradient. We shall consider all solutions in the usual Sobolev space
\begin{equation}
\label{FuncSpace}
H_0^\sigma(\Omega),\quad\text{with norm }\|u\|_{H_0^\sigma(\Omega)}=\|D^\sigma u\|_{L^2(\R^N)^N},\quad 0<\sigma<1,
\end{equation}
where here and whenever necessary, we still denote by $u$ the extension of $u$ by zero outside of $\Omega$. By the Sobolev-Poincar\'e inequality, this norm is equivalent to the usual Hilbertian norm induced from $L^{\sigma,2}(\R^N)=W^{\sigma,2}(\R^N)=H^\sigma(\R^N)$, $0<\sigma<1$ in the closure of the Cauchy sequences of smooth functions with compact support in $\Omega$, i.e., in  $\C^\infty_0(\Omega)$ (see \cite{ShiehSpector2015}).

For nonnegative functions $g\in L^\infty({\R^N})$, we consider the nonempty convex sets of the type
\begin{equation}
\label{kapa}
\K_g^\sigma=\big\{v\in H^\sigma_0(\Omega):|D^\sigma v|\le g\text{ a.e. in }\R^N\big\}.
\end{equation}

Let $f\in L^1(\Omega)$ and $A:\R^N\rightarrow\R^{N\times N}$ be a measurable, bounded and positive definite matrix. We shall consider, in Section \ref{2}, the well-posedness of the variational inequality
\begin{equation}
\label{iv}
u\in\K_g^\sigma:\qquad \int_{\R^N}A D^\sigma u\cdot D^\sigma (v-u)\ge \int_{\Omega} f(v-u),\qquad\forall v\in \K_g^\sigma.
\end{equation}

In particular, we obtain precise estimates for the continuous dependence of the solution $u$ with respect to $f$ and $g$, and so we extend well-known results for the classical case $\sigma=1$ (see \cite{RodriguesSantos2019} and its references).

Extending the result of \cite{AzevedoSantos2017} for the gradient ($\sigma=1$) case, we prove in Section \ref{3} the existence of generalised Lagrange multipliers for the $\sigma$-gradient constrained problem. More precisely, we show the existence of $(\lambda,u)\in L^\infty(\R^N)^{\bs '}\times \Upsilon_\infty^\sigma(\Omega)$ such that
\begin{subequations}
\begin{align}
\label{lm1}
&\bs\langle \lambda D^\sigma u,D^\sigma v\bs\rangle_{\big(L^\infty(\R^N)^N\big)^{\bs '}\times L^\infty(\R^N)^N}+\int_{\R^N} AD^\sigma u\cdot D^\sigma v=\int_{\Omega}f v,\qquad\forall v\in \Upsilon_\infty^\sigma(\Omega),\\
\label{lm2}
&|D^\sigma u|\le g\ \text{ a.e. in }{\R^N},\qquad \lambda\ge 0 \text { and }  \lambda(|D^\sigma u|-g)=0\  \text { in }L^\infty(\R^N)^{\bs '}
\end{align}
\end{subequations}
and, moreover, $u$ solves \eqref{iv}.
Here, for each $\sigma$, we have set
\begin{equation}
\label{Ginfty}
\Upsilon_\infty^\sigma(\Omega)=\big\{\upsilon\in H^\sigma_0(\Omega): D^\sigma \upsilon\in L^\infty(\R^N)^N\big\},\qquad 0<\sigma<1,
\end{equation}
and 
$$\langle\lambda\bs\alpha,\bs\beta\rangle_{(L^\infty(\R^N)^N){\bs '}\times L^\infty(\R^N)^N}=\langle\lambda,\bs\alpha\cdot\bs\beta\rangle_{L^\infty(\R^N){\bs '}\times L^\infty(\R^N)}\quad\forall\lambda\in L^\infty(\R^N){\bs '}\ \forall\bs\alpha,\bs\beta\in L^\infty(\R^N)^N.$$

Finally, in the Section \ref{4} we consider the solvability of solutions to quasi-variational inequalities corresponding to \eqref{iv} when the threshold $g=G[u]$ and therefore also the convex set \eqref{kapa} depend on the solution $u\in\K_{G[u]}^\sigma$. We give sufficient conditions on the nonlinear and nonlocal operator $v\mapsto G[v]$ to obtain the existence of at least one solution $u$ of \eqref{iv} with $\K_g^\sigma$ replaced by $K_{G[u]}^\sigma$, by compactness methods, as in \cite{KunzeRodrigues2000} for the case $\sigma=1$. In a special case, when $G[u](x)=\Gamma(u)\varphi(x)$ is strictly positive and separates variables with a  Lipschitz functional $\Gamma:L^2 (\Omega)\rightarrow\R^+$, we adapt an idea of \cite{HintermullerRautenberg2013} (see also \cite{RodriguesSantos2019}) to obtain, by a contraction principle, the existence and uniqueness of the solution of the quasi-variational inequality under the ``smallness" of the product of $f$ with the Lipschitz constant of $\Gamma$ and the inverse of its positive lower bound.

\section{The variational inequality with $\sigma$-gradient constraint}\label{2}

For some $a_*, a^*>0$, let $A=A(x):\R^N\rightarrow\R^{N\times N}$ be a bounded and measurable matrix, not necessarily symmetric, such that, for a.e. $x\in\R^N$ and all $\xi,\eta\in\R^N$:
\begin{equation}
\label{aa}
a_*|\xi|^2\le A(x)\xi\cdot\xi\quad\text{and}\quad A(x)\xi\cdot\eta \le a^*|\xi||\eta|.
\end{equation}

Fixed $\nu>0$, we define
\begin{equation}
\label{linfnu}L^\infty_\nu({\R^N})=\big\{v\in L^\infty({\R^N}):v(x)\ge \nu>0\text { a.e. }x\in{\R^N}\big\}.
\end{equation}

Indeed, we recall (see for instance \cite{DemengelDemengel2012})
the embedding for the fractional Sobolev spaces $0<\sigma\le 1$, $1<p<\infty$:
\begin{subequations} 
\begin{align}
\label{sob1}
&W^{\sigma,p}(\Omega)\hookrightarrow L^q(\Omega),\quad\text{ for every }q\le \tfrac{Np}{N-\sigma p},\ \text{ if }\sigma p<N,\\
\label{sob2}
&W^{\sigma,p}(\Omega)\hookrightarrow L^q(\Omega),\quad\text{ for every }q< \infty,\ \text{ if }\sigma p=N,\\
\label{sob3}
&W^{\sigma,p}(\Omega)\hookrightarrow L^\infty(\Omega)\cap \C^{0,\beta}(\overline{\Omega}),\quad\text{ for every }0<\beta\le\sigma -\tfrac{N}p,\ \text{ if }\sigma p>N,
\end{align}
\end{subequations}
with continuous embedding, which are also compact if also $q<\tfrac{Np}{N-\sigma p}$ in \eqref{sob1} and $\beta<\sigma -\tfrac{N}p$ in \eqref{sob3}. These embedding properties extend to the the Bessel potential space $L^{\sigma,p}(\R^N)$ exactly with the same values of   $\sigma, p, q$ and $\beta$ (see \cite[Thm. 2.2]{ShiehSpector2015}). In particular, we have
\begin{equation}
\label{hinfty}H^\sigma_0(\Omega)\hookrightarrow L^{2^*}(\Omega)\quad\text{and}\quad L^{2^\#}(\Omega)\hookrightarrow H^{-\sigma}(\Omega)=\big(H^\sigma_0(\Omega)\big)^{\bs '},\ 0<\sigma<1,
\end{equation}
where we set $2^*=\frac{2N}{N-2\sigma}$ and $2^{\#}=\frac{2N}{N+2\sigma}$ when $\sigma<\frac{N}2$, and if $N=1$ we denote $2^*=q$, $2^{\#}=q'=\frac{q}{q-1}$ when $\sigma=\frac{1}2$ and $2^*=\infty$, $2^\#=1$ when $\sigma>\frac12$. 

%\subset \Upsilon_\infty^\sigma(\Omega)

For any $g\in L^\infty_\nu({\R^N})$, the  convex set $\K_g^\sigma$ is non-empty and closed and we have the following auxiliary result.
	\begin{prop} \label{dsigmalq} For any $g\in L^\infty_\nu({\R^N})$, we have the following inclusions:
		\begin{equation}
		\label{inclusions}
		\K_g^\sigma\subset \Upsilon^\sigma_\infty(\Omega)\subset \C^{0,\beta}(\overline{\Omega})\subset L^\infty(\Omega),
		\end{equation}
		for all $0<\beta<\sigma$, where $\C^{0,\beta}(\overline{\Omega})$ is the space of H\"older continuous functions with exponent $\beta$
		and the estimate
		\begin{equation}
			\label{maj}\| u\|_{L^\infty(\Omega)}\le \kappa\| u\|_{H^\sigma_0(\Omega)},\qquad\text{for any }u\in\Upsilon_\infty^\sigma(\Omega),
			\end{equation}
			holds, where $\kappa>0$ depends on $\Omega$ through the Sobolev imbeddings and on $\|D^\sigma u\|_{L^\infty(\R^N)^N}$.
\end{prop}
\begin{proof}  For  $u\in\Upsilon_\infty^\sigma(\Omega)$, as  $H^\sigma_0(\Omega)\subset L^{2^*}(\Omega)$ and  recalling that $u=0$ on $\R^N\setminus B_R$, with the estimate
$$\int_{R^N}|D^\sigma u|^p\le\|D^\sigma u\|_{L^\infty(\R^N)^N}^{p-2}\int_{\R^N}|D^\sigma u|^2,\qquad\forall\, 2<p<\infty,$$
we obtain $u\in L^{\sigma,2^{*} }(\R^N)\subset L^{2^{**}}(\R^N)$, where $2^{**}=\frac{2^*N}{N-2^*\sigma}$ if $2^*\sigma<N$ and $2^{**}$ is any $q$ finite if $2^*\sigma\ge N$,
by using Sobolev imbeddings (see \cite[Thm. 2.2]{ShiehSpector2015}). Iterating  with a bootstrap argument, we obtain $u\in L^{\sigma,q}(\R^N)^N\subset \C^{0,\beta}(\overline\Omega)$, for any $q>\tfrac{N}\sigma$, with $\beta=\sigma-\tfrac{N}q$.
\end{proof} 

% Here we are also assuming that $\Omega\subset\R^N$ is an open, bounded domain with Lipschitz boundary, and we may conclude \eqref{inclusions} from \eqref{sob1}-\eqref{sob3} by using a bootstrap argument.

Therefore, in the right hand side of the variational inequality \eqref{iv}, for $g_i\in L^\infty(\R^N)$, we can take $f_i\in L^1(\Omega)$, and the first two theorems give continuous dependence results with precise estimates for two different problems with $i=1,2$:
\begin{equation*}
(2.7)_i\qquad\qquad u_i\in\K_{g_i}^\sigma:\qquad \int_{\R^N} AD^\sigma u_i\cdot D^\sigma(v-u_i)\ge\int_\Omega f_i(v-u_i),\quad\forall v\in\K_{g_i}^\sigma.\qquad\qquad
\end{equation*}

\setcounter{equation}{7}

\begin{thm}\label{thm21}
	Under the assumptions \eqref{aa}, for each $f_i\in L^1(\Omega)$ and each $g_i\in L^\infty(\R^N)$, $g_i\ge0$, there exists a unique solution $u_i$ to $(2.7)_i$ such that
	\begin{equation}
	\label{kapai}
	u_i\in\K_{g_i}^\sigma\cap\C^{0,\beta}(\overline{\Omega}),\quad\text{ for all }0<\beta<\sigma.
	\end{equation}
	
	When $g_1=g_2$, the solution map $L^1(\Omega)\ni f\mapsto u\in H^\sigma_0(\Omega)$ is  Lipschitz continuous, i.e., for $C_1=\kappa/a_*>0$, where $\kappa>0$ is the constant \eqref{maj}, we have
	\begin{equation}
	\label{fs}\|u_1-u_2\|_{H^\sigma_0(\Omega)}\le C_1\|f_1-f_2\|_{L^1(\Omega)}.
	\end{equation}
	
	Moreover, if in addition $f_i\in L^{2^{\#}}(\Omega)$, $i=1,2$, with $2^\#$ defined in \eqref{hinfty} and $g_1=g_2$, then $L^{2^\#}(\Omega)\ni f\mapsto u\in H^\sigma_0(\Omega)$ is  Lipschitz continuous:
	\begin{equation}
	\label{lip}
	\|u_1-u_2\|_{H^\sigma_0(\Omega)}\le C_\#\|f_1-f_2\|_{L^{2^\#}(\Omega)},
	\end{equation}
	for $C_\#=C_*/{a_*}>0$, where $C_*$ is the constant of the Sobolev embedding $H^\sigma_0(\Omega)\hookrightarrow L^{2^*}(\Omega)$.
\end{thm}
\begin{proof}
	Suppose that $f_i\in L^{2^\#}(\Omega)\subset H^{-\sigma}(\Omega)$. Since the assumption \eqref{aa} implies that $A$ defines a continuous bilinear and coercive form over $H^\sigma_0(\Omega)$, the existence and uniqueness of the solution $u_i\in\K^\sigma_{g_i}$ to $(2.7)_i$ is an immediate consequence of the Stampacchia Theorem (see, for instance, \cite[p. 95]{Rodrigues1987}), and \eqref{kapai} follows from \eqref{inclusions}.
	
	With our notation \eqref{FuncSpace}, the estimate \eqref{lip} follows easily from $(2.7)_i$ with $g_1=g_2$ and $v=u_j$ ($i,j=1,2$, $i\neq j$) from
	\begin{equation*}
	a_*\|\overline u\|_{H^\sigma_0(\Omega)}^2\le\int_{\R^N} AD^\sigma \overline u\cdot D^\sigma\overline u\le\|\overline f\|_{L^{2^\#}(\Omega)}\|\overline u\|_{L^{2^*}(\Omega)}\le C_*\|\overline f\|_{L^{2^\#}(\Omega)}\|\overline u\|_{H^\sigma_0(\Omega)},
	\end{equation*}
	where we have set $\overline u=u_1-u_2$ and $\overline f=f_1-f_2$.
	
	By \eqref{maj}, letting $\kappa>0$ be such that
	$$\|\overline u\|_{L^\infty(\Omega)}\le \kappa\|\overline u\|_{H^\sigma_0(\Omega)},$$
	we may easily conclude the estimate \eqref{fs} with $C_1=\kappa/a_*$ for $f_1, f_2\in L^{2^\#}(\Omega)\subset L^1(\Omega)$ from (1.5)$_i$ and
	$$a_*\|\overline u\|_{H^\sigma_0(\Omega)}^2\le\|\overline f\|_{L^1(\Omega)}\|\overline u\|_{L^\infty(\Omega)}\le \kappa\|\overline f\|_{L^1(\Omega)} \|\overline u\|_{H^\sigma_0(\Omega)}.$$
	
	Finally, the solvability of $(2.7)_i$ for $f_i$ only in $L^1(\Omega)$ can be easily obtained by taking an approximating sequence of $f_i^n\in L^{2^\#}(\Omega)$ such that $f_i^n\underset{n}\rightarrow f_i$ in $L^1(\Omega)$ and using \eqref{fs} for that (Cauchy) sequence. The proof is complete.
\end{proof}

\begin{rem}
	As in \cite{ShiehSpector2015} it is possible to extend the variational inequality with $\sigma$-gradient to arbitrary open domains $\Omega\subset\R^N$ with a generalised Dirichlet data $\varphi\in H^\sigma(\R^N)$ such that $I_{1-\sigma}*\varphi$ is well-defined and $D^\sigma\varphi\in L^\infty(\R^N)$. This would require in the definition \eqref{kapa} of $\K_g^\sigma$ to replace $H^\sigma_0(\Omega)$ by the space
	$$H^\sigma_\varphi=\big\{v\in H^\sigma(\R^N):v=\varphi\text{ a.e. in }\R^N\setminus\Omega\big\}$$
	and, in addition, technical compatibility assumptions on $\varphi$ and $g$ to guarantee that the new $\K_g^\sigma\not=\emptyset$.
\end{rem}

\begin{rem}
	It is well-known that if, in addition, $A$ is symmetric, i.e. $A=A^T$, the variational inequality \eqref{iv} corresponds (and is equivalent) to the optimisation problem (see, for instance, \cite{Rodrigues1987})
	$$u\in\K_g^\sigma:\qquad \mathcal J(u)\le\mathcal J(v),\qquad\forall v\in\K_g^\sigma,$$
	where $\mathcal J:\K_g^\sigma\rightarrow\R$ is the convex functional
	$$\mathcal J(v)=\frac12\int_{\R^N}AD^\sigma v\cdot D^\sigma v-\int_\Omega fv.$$
\end{rem}

\begin{thm}\label{thm22}
	Under the framework of the previous theorem, when $f_1=f_2\in L^1(\Omega)$, the solution map
	$$L^\infty_\nu(\R^N)\ni g\mapsto u\in H^\sigma_0(\Omega)$$
	is $\frac12$-H\"older continuous, i.e., there exists $C_\nu>0$ such that
	\begin{equation}
	\label{cdg}\|u_1-u_2\|_{H^\sigma_0(\Omega)}\le C_\nu\|g_1-g_2\|^\frac12_{L^\infty(\R^N)}.
	\end{equation}
\end{thm}
\begin{proof}
	Let $\eta=\|g_1-g_2\|_{L^\infty(\R^N)}$ and, for $i,j=1,2$, $i\neq j$, notice that
	$$u_{i_j}=\frac{\nu}{\nu+\eta}u_i\in\K^\sigma_{g_j},$$
	if $u_i$ denotes the unique solution of $(2.15)_i$ to $g_i$ and $f_i$.
	
	Observe that for $i=1,2,$
	$$|u_i-u_{i_j}|\le\frac\eta{\nu+\eta}|u_i|\le \frac{\eta}\nu |u_i|\quad \text{and} \quad  |D^\sigma(u_i-u_{i_j})|\le\frac\eta{\nu+\eta}|D^\sigma u_i|\le \frac{\eta}\nu |D^\sigma u_i|.$$
	
	Hence, letting $v=u_{i_j}$ in  $(2.7)_j$ and using \eqref{aa} we get \eqref{cdg} from
	\begin{multline*}
	a_*\|u_1-u_2\|^2_{H^\sigma_0(\Omega)}\le\int_{\R^N} AD^\sigma(u_1-u_2)\cdot D^\sigma(u_1-u_2)\\
	\le \int_{\R^N}AD^\sigma u_1\cdot D^\sigma(u_{2_1}-u_2)+
	\int_{\R^N}AD^\sigma u_2\cdot D^\sigma(u_{1_2}-u_1)+\int_\Omega f\big((u_1-u_{1_2})+(u_2-u_{2_1})\big)\\
\le2\tfrac{a^*\eta}{\nu}\|D^\sigma u_1\|_{L^2(\R^N)}\|D^\sigma u_2\|_{L^2(\R^N)}+\tfrac\eta\nu\|f\|_{L^1(\Omega)}(\|u_1\|_{L^\infty(\Omega)}+\|u_2\|_{L^\infty(\Omega)})\\
\le C^{'}_\nu\|g_1-g_2\|_{L^\infty(\R^N)},
	\end{multline*}
	with $C^{'}_\nu=2\kappa^2\|f\|^2_{L^1(\Omega)}(a^*+a_*)/a_*^2\nu>0$, by using Theorem \ref{thm21} and the constant $\kappa$ is defined by \eqref{maj}.
\end{proof}

\begin{rem}
	Using the trick of the above proof, if $g_n\underset{n}{\rightarrow} g$ in $L^\infty(\R^N)$	 for a sequence $g_n\in L^\infty_\nu(\R^N)$, it is clear that, for any $w\in\K_g^\sigma$ we can choose $w_n\in\K_{g_n}^\sigma$ such that $w_n\underset{n}{\rightarrow} w$ in $H^\sigma_0(\Omega)$. On the other hand, also for any sequence $w_n\underset{n}{\lraup}w$ in $H^\sigma_0(\Omega)$-weak, with each $w_n\in\K_{g_n}^\sigma$, $g_n\underset{n}{\rightarrow} g$ in $L^\infty(\R^N)$ implies that also $w\in\K_g^\sigma$. These two conditions determine that if $g_n\underset{n}{\rightarrow} g$ in $L^\infty_\nu(\R^N)$ then the respective convex sets $\K_{g_n}^\sigma$ converge in the Mosco sense to $\K_g^\sigma$. An open question is to extend this convergence to the case $0<\sigma<1$, by dropping the strict positivity condition on $g_n$ and $g$, as in \cite{AzevedoSantos2004} for $\sigma=1$.
\end{rem} 
\section{Existence of Lagrange multipliers}\label{3}

In this section we prove the existence of solution of the problem \eqref{lm1}-\eqref{lm2}.

For $\eps\in(0,1)$ and denoting $\widehat k_\eps=\widehat k_\eps(D^\sigma u^\eps)=k_\eps(|D^\sigma u^\eps|-g)$ for simplicity, we define a family of approximated quasi-linear problems
\begin{equation}\label{ap}
\int_{\R^N}\big(\widehat k_\eps(D^\sigma u^\eps)D^\sigma u^\eps+AD^\sigma u^\eps\big)\cdot D^\sigma v=\int_{\Omega}f v\qquad\forall v\in H^\sigma_0(\Omega)
\end{equation}
where  $k_\eps:\R\rightarrow\R$ is defined by
$$ k_\eps(s)=0\text{ for }s<0,\qquad k_\eps(s)=e^\frac{s}\eps-1\text{ for }0\le s\le\tfrac1\eps\qquad k_\eps(s)=e^\frac1{\eps^2}-1\text{ for }s>\tfrac1\eps.$$

\begin{prop}\label{tapprox}
Suppose that $g\in L^\infty_\nu(\R^N)$,  $f\in L^{2^\#}(\Omega)$ and $A:{\R^N}\rightarrow\R^{N\times N}$ is a measurable, bounded and positive definite matrix. Then the quasi-linear problem \eqref{ap} has a unique solution $u^\eps\in H^\sigma_0(\Omega)$.
\end{prop}
\begin{proof} The operator $B_\eps:H^\sigma_0(\Omega)\rightarrow H^{-\sigma}(\Omega)$ defined by
	$$\langle B_\eps v,w\rangle =\int_{\R^N}\big(\widehat k_\eps (D^\sigma v)D^\sigma v+AD^\sigma v\big)\cdot D^\sigma w$$
is bounded,  strongly monotone, coercive and hemicontinuous, so problem \eqref{ap} has a unique solution (see, for instance, \cite{Lions1969}).
\end{proof}

\begin{lemma}\label{estimates-l1}
If $g\in L^\infty_\nu({\R^N})$,  $f\in L^{2^\#}(\Omega)$,  $A:{\R^N}\rightarrow\R^{N\times N}$ is a measurable, bounded and positive definite matrix and $1\leq q<\infty$, there exist  positive constants $C$ and $C_q$ such that, for $0<\varepsilon<1$, setting $\widehat k_\eps=k_\eps(|D^\sigma u^\eps|-g)$, the solution $u^\varepsilon$ of
	the approximated problem \eqref{ap} satisfies
	\begin{subequations}
	\begin{align}
	\label{LM_grad2}
	\|D^\sigma u^\eps\|_{L^2(\R^N)}&\le C,\\
	\label{LM_lemma_0}
	\|\widehat k_\varepsilon| D^\sigma u^\varepsilon|^2\|_{L^1({\R^N})} &\leq C,\\
	\label{LM_lemma_1}
	\|\widehat k_\varepsilon\|_{L^1({\R^N})} &\leq C,\\
	\label{LM_lemma_4}
	\|\widehat  k_\eps D^\sigma   u^\varepsilon\|_{( L^\infty({\R^N})^N)^{\bs '}} &\leq C,\\
		\label{LM_lemma_5}
	\|\widehat k_\eps \|_{L^\infty({\R^N})^{\bs '}} &\leq C\\
	\label{LM_lemma_2}
	\|D^\sigma  u^\varepsilon\|_{ L^q(\R^N)^N}&\le C_q.
	\end{align}
	\end{subequations}
\end{lemma}
\begin{proof}
Using $u^\eps$ as test function in \eqref{ap}, we get
\begin{align*}\int_{\R^N}\big(\widehat k_\eps+a_*\big)|D^\sigma u^\varepsilon|^2&\le\int_{\R^N}\widehat k_\eps |D^\sigma u^\varepsilon|^2 +AD^\sigma u^\varepsilon\cdot D^\sigma u^\varepsilon  \\
&=\int_\Omega f u^\eps\le \frac{C_\#^{\ 2}}{2a_*}\|f\|^2_{L^{2^\#}(\Omega)}+\frac{a_*}2\|D^\sigma u^\eps\|^2_{L^2({\R^N})^N},
\end{align*}
since $A\xi\cdot\xi\ge a_*|\xi|^2$ for any $\xi\in\R^N$ by the assumptions on $A$. But $\widehat k_\eps \ge0$ and so
$$\int_{\R^N}\widehat k_\eps|D^\sigma u^\eps|^2+\frac{a_*}2\int_{{\R^N}}|D^\sigma u^\eps|^2\le \frac{C_\#^{\ 2}}{2a_*}\|f\|^2_{L^{2^\#}(\Omega)},$$
concluding then  \eqref{LM_grad2} and \eqref{LM_lemma_0}. 

Observing that the function $\varphi_\eps=\widehat k_\eps\,(t^2-g^2)+g^2 \widehat k_\eps\ge \nu^2\widehat k_\eps$  and using \eqref{LM_lemma_0}, there exists a positive constant $C$ independent of $\eps$ such that
$$\nu^2\int_{\R^N}	\widehat k_\eps \le C.$$
This implies the uniform boundedness of $\widehat k_\eps $ in $L^1({\R^N})$ and also in $L^\infty({\R^N})^{\bs '}$, i.e., \eqref{LM_lemma_1} and \eqref{LM_lemma_5} respectively.

To prove \eqref{LM_lemma_4}, it is enough to notice that, for $ \bs\beta\in L^\infty(\Omega)^N$, 
\begin{align*}
\|\widehat k_\eps D^\sigma u^\eps\|_{(L^\infty({\R^N})^N)^{\bs '}}&= \sup_{\bs\beta\in L^\infty({\R^N})^N}\int_{\R^N} \widehat k_\eps D^\sigma u^\varepsilon\cdot \bs\beta\
&\le \Big(\int_{\R^N} \widehat k_\eps |D^\sigma u^\eps|^2\Big)^\frac12\Big(\int_{\R^N} \widehat k_\eps |\bs\beta|^2 \Big)^\frac12\\
&\le C\| \bs\beta\|_{L^\infty({\R^N})^N}.
\end{align*}

 Because for  $t-g>0$ we have $\widehat k_\eps(t-g)\ge \frac 1{m!}(t-g)^m$, for any $m\in\N$, then  using \eqref{LM_lemma_1}  we obtain $\|D^\sigma u^\eps\|_{L^q(B_R)}\le C_{q,R}$, for any $q\in(1,\infty)$ and $R\in\R^+$, (for details see, for instance \cite{MirandaRodriguesSantos2012},  replacing $Q_T$ by $B_R$).  Arguing as in Proposition \ref{dsigmalq}, we conclude the proof.
\end{proof}

\begin{prop}\label{H1}
	For  $g\in L^\infty_\nu({\R^N})$, $f\in L^{2^\#}(\Omega)$ and $A:{\R^N}\rightarrow\R^{N\times N}$  a measurable, bounded and positive definite matrix, the family  $\{u^\varepsilon\}_\varepsilon$ of solutions of the approximated problems \eqref{ap} converges weakly in $H^\sigma _0(\Omega)$ to the solution of the variational inequality \eqref{iv}.
\end{prop}
\begin{proof} The uniform boundedness of $\{u^\eps\}_\eps$ in $H^\sigma _0(\Omega)$ implies that, at least for a subsequence,
\begin{equation}
\label{uepfraco}
u^\eps\underset{\eps\rightarrow0}{\lraup}u\quad\text{ in } H^\sigma _0(\Omega).
\end{equation}

For $v\in\K_g^\sigma$ we have, since $\widehat k_\eps>0$ when $|D^\sigma u_\eps|>g\ge |D^\sigma v|$,
$$\widehat k_\eps D^\sigma u^\eps\cdot D^\sigma( v-u^\eps)\le\widehat k_\eps|D^\sigma u^\eps|(|D^\sigma v|-|D^\sigma u^\eps|)\le 0$$
and so, testing the first equation of \eqref{ap} with $v-u^\eps$, we get
$$\int_{{\R^N}}AD^\sigma u^\eps\cdot D^\sigma( v-u^\eps)\ge\int_\Omega f(v-u^\eps).$$

But
\begin{align*}
\int_{{\R^N}}AD^\sigma u^\eps\cdot D^\sigma( v-u^\eps)&=\int_{{\R^N}}AD^\sigma (u^\eps-v)\cdot D^\sigma( v-u^\eps)+\int_{{\R^N}}AD^\sigma v\cdot D^\sigma( v-u^\eps)\\
&\le \int_{{\R^N}}AD^\sigma v\cdot D^\sigma( v-u^\eps)
\end{align*}
So, utilizing the weak convergence $u^\eps\underset{\eps\rightarrow0}{\lraup} u$ in $H^\sigma_0(\Omega)$,
$$\int_{\R^N}AD^\sigma v\cdot D^\sigma (v-u)\ge\int_\Omega f(v-u).$$
Let $w\in\K_g^\sigma$ and setting $v=u+\theta(w-u)$, then $v\in\K_g^\sigma$ for any $\theta\in(0,1]$ and  we get
$$\theta\int_{\R^N}AD^\sigma (u+\theta(w-u))\cdot D^\sigma (w-u)\ge\theta\int_\Omega f(w-u).$$
Dividing this inequality by $\theta$ and letting $\theta\rightarrow0$, we obtain \eqref{iv}. The proof is concluded if we show that $u\in\K_g^\sigma$. Indeed we  split $\R^N$ in three subsets
$$U_\eps=\big\{|D^\sigma u^\eps|-g\le\sqrt\eps\big\}, \quad V_\eps=\big\{\sqrt\eps\le|D^\sigma u^\eps|-g\le\tfrac1\eps\big\}, \quad W_\eps=\big\{|D^\sigma u^\eps|-g>\tfrac1\eps\big\} $$
and, following the steps in \cite{MirandaRodriguesSantos2012}, we conclude, for arbitrary $R>0$, that
\begin{align*}
\int_{B_R}\big(|D^\sigma u|-g\big)^+&\le\varliminf_{\eps\rightarrow0}\int_{ B_R}\big(\big(|D^\sigma u^\eps|-g\big)\vee0)\wedge\frac1\varepsilon\\
&=\varliminf_{\eps\rightarrow0}\left(\int_{U_\eps{\cap B_R}}\big(|D^\sigma u^\eps|-g\big)\vee0+\int_{V_\eps{\cap B_R}}\big(|D^\sigma u^\eps|-g\big)+\int_{W_\eps{\cap B_R}}\tfrac1\eps\right)\\
&\le\varliminf_{\eps\rightarrow0}\left(\sqrt\eps|B_R|+\||D^\sigma u^\eps|-g\|_{L^2({B_R})}\,|V_\eps{\cap B_R}|^\frac12+\int_{W_\eps{\cap B_R}}\tfrac1\eps\right)\underset{\eps\rightarrow0}{\longrightarrow}0,
\end{align*}
because
$$|V_\eps{\cap B_R}|\le\int_{V_\eps{\cap B_R}}\tfrac{\widehat k_\eps+1}{e^{\frac1{\sqrt{\eps}}}}\le C{_R}e^{\tfrac{-1}{\sqrt{\eps}}}\underset{\eps\rightarrow0}{\longrightarrow}0\   \text{and}\  \int_{W_\eps{\cap B_R}}\tfrac1\eps= \tfrac1\eps\int_{W_\eps{\cap B_R}}\tfrac{\widehat k_\eps+1}{e^\frac1{\eps^2}}\le \tfrac{C{_R}}{\eps}e^{-\frac1{\eps^2}}\underset{\eps\rightarrow0}{\longrightarrow}0.$$
So $|D^\sigma u|\le g$ a.e. in ${ B_R}$, which means that $u\in\K_g^\sigma$  because $\R^N=\displaystyle\cup_{k\in\N}B_k$.

The uniqueness of solution of the variational inequality \eqref{iv} implies that the whole sequence $\{u^\eps\}_\eps$ converges to $u$ in $H^\sigma_0(\Omega)$.
\end{proof}

\begin{thm}\label{thm_lm} If $g\in L^\infty_\nu({\R^N})$,  $f\in L^{2^\#}(\Omega)$ and $A:\R^N\rightarrow\R^{N\times N}$ is a measurable, bounded and positive definite matrix, then  problem \eqref{lm1}-\eqref{lm2} has a solution
$$(\lambda, u)\in L^{\infty}({\R^N})^{\bs '}\times  \Upsilon_\infty^\sigma(\Omega).$$
\end{thm}
\begin{proof}
By estimates \eqref{LM_lemma_4} and \eqref{LM_lemma_5} and the Banach-Alaoglu-Bourbaki theorem we have, at least for a subsequence,
\begin{equation*}
\widehat k_\eps D^\sigma  u^\varepsilon\underset{\varepsilon\rightarrow0}{\lraup}\Lambda\ \text{weak in } \big( L^\infty({\R^N})^N\big)^{\bs'}
\end{equation*}                                                                                                                
and
\begin{equation*}
\widehat k_\eps \underset{\varepsilon\rightarrow0}{\lraup}\lambda\ \text{weak in } L^\infty({\R^N}){\bs '}.
\end{equation*}

For $v\in H^\sigma_0(\Omega)$, since
\begin{equation}
\label{mleps}
\int_{\R^N}\big(\widehat k_\eps D^\sigma u^\eps+AD^\sigma u^\eps\big)\cdot D^\sigma v=\int_\Omega f v, 
\end{equation}
we obtain, letting $\eps\rightarrow0$ with $v \in  \Upsilon_\infty^\sigma(\Omega)$,
\begin{equation}\label{mlepslim}
\bs\langle \Lambda,D^\sigma v\bs\rangle+\int_{{\R^N}}AD^\sigma u\cdot D^\sigma v=\int_\Omega fv.
\end{equation}

Taking $v=u^\eps$ in  \eqref{mleps} we get
\begin{equation}\label{eq}
\int_{{\R^N}}\widehat k_\eps |D^\sigma  u^\eps|^2+\int_{{\R^N} }AD^\sigma u^\eps\cdot D^\sigma u^\eps=\int_\Omega fu^\eps
\end{equation}

Observe first that 
\begin{multline}\label{3.7}
\int_{\R^N} AD^\sigma (u^\eps-u)\cdot D^\sigma u^\eps=\int_{\R^N} AD^\sigma (u^\eps-u)\cdot D^\sigma (u^\eps-u)\\
+\int_{\R^N} AD^\sigma (u^\eps-u)\cdot D^\sigma u\ge 
\int_{\R^N} AD^\sigma (u^\eps-u)\cdot D^\sigma u
\end{multline}
and therefore
$$\int_{\R^N}AD^\sigma u\cdot D^\sigma u\le\varliminf_{\eps\rightarrow0}\int_{\R^N} AD^\sigma u^\eps\cdot D^\sigma u^\eps.$$
So, from\eqref{eq} and \eqref{mlepslim} with $v=u$,
\begin{align*}
\varliminf_{\eps\rightarrow0}\int_{{\R^N}}\widehat k_\eps |D^\sigma  u^\eps|^2+\int_{\R^N} AD^\sigma u\cdot D^\sigma u&\le\varliminf_{\eps\rightarrow0}\Big(\int_{\R^N}\widehat k_\eps |D^\sigma  u^\eps|^2+\int_{{\R^N}}AD^\sigma u^\eps\cdot D^\sigma u^\eps\Big)\\
&=\int_{\Omega }fu=\bs\langle \Lambda,D^\sigma u\bs\rangle+\int_{{\R^N} }AD^\sigma u\cdot D^\sigma u
\end{align*}                                                         
and then 
$$\varliminf_{\eps\rightarrow0}\int_{{\R^N}}\widehat k_\eps |D^\sigma  u^\eps|^2\le \bs\langle \Lambda,D^\sigma u\bs\rangle.$$
Using $\widehat  k_\eps(|D^\sigma u^\eps|^2-g^2)\ge0$, 
%\eqref{eq} and \eqref{3.7}, 
we obtain
\begin{equation*}
\bs\langle\Lambda,D^\sigma  u\bs\rangle\ge\varliminf_{\eps\rightarrow0}\int_{{\R^N}}\widehat k_\eps|D^\sigma  u^\eps|^2\ge\lim_{\eps\rightarrow0}\int_{{\R^N}}\widehat k_\eps g^2=\langle\lambda, g^2\rangle \ge\langle\lambda, |D^\sigma u|^2\rangle.
\end{equation*}

We also have
\begin{align*}
0\le\varliminf_{\eps\rightarrow0}\int_{\R^N} \widehat k_\eps|D^\sigma (u^\eps-u)|^2&=\varliminf_{\eps\rightarrow0}\int_{\R^N} \widehat k_\eps|D^\sigma u^\eps|^2-2\lim_{\eps\rightarrow0}\int_{\R^N} \widehat k_\eps D^\sigma u^\eps\cdot D^\sigma u\\
&\hspace{0,5cm}+\lim_{\eps\rightarrow0}\int_{\R^N}\widehat k_\eps|D^\sigma u|^2\\
&\le \bs\langle \Lambda,D^\sigma u\bs\rangle-2\bs\langle \Lambda,D^\sigma u\bs\rangle+\langle\lambda,|D^\sigma u|^2\rangle\\
&=-\bs\langle \Lambda,D^\sigma u\bs\rangle+\langle\lambda,|D^\sigma u|^2\rangle,
\end{align*}
and therefore we conclude
$$\bs\langle \Lambda,D^\sigma u\bs\rangle=\langle\lambda,|D^\sigma u|^2\rangle\quad\text{and}\quad
\varliminf_{\eps\rightarrow0}\int_{\R^N} \widehat k_\eps|D^\sigma (u^\eps-u)|^2=0.$$

Given $v\in\K_g$, we have
\begin{multline}\label{strong}
\varliminf_{\eps\rightarrow0}\Big|\int_{{\R^N}}\widehat k_\eps D^\sigma (u^\eps-u)\cdot D^\sigma  v\Big|\\
\le\varliminf_{\eps\rightarrow0}\left(\int_{\R^N}\widehat k_\eps |D^\sigma (u^\eps-u)|^2\right)^\frac12\|\widehat k_\eps \|^\frac12_{L^1(\R^N)}\|D^\sigma v\|_{L^\infty(\R^N)}=0,
\end{multline}
because, by estimate \eqref{LM_lemma_1}, $\widehat k_\eps $ is uniformly bounded in $L^1(\R^N)$. So, for any $v\in\K_g$,
\begin{multline*}
\int_\Omega fv=\varliminf_{\eps\rightarrow0}\int_{{\R^N}}(\widehat k_\eps +A) D^\sigma u^\eps\cdot D^\sigma v
=\varliminf_{\eps\rightarrow0}\Big(\int_{{\R^N}}(\widehat k_\eps +A)
D^\sigma (u^\eps-u)\cdot D^\sigma v\\
+\lim_{\eps\rightarrow0}\int_{{\R^N}}(\widehat k_\eps +A) D^\sigma u\cdot D^\sigma v\Big)
=\bs\langle\lambda D^\sigma u, D^\sigma v\bs\rangle+\int_{\R^N} AD^\sigma u\cdot D^\sigma v,
\end{multline*}
concluding the proof of \eqref{lm1}.

 Since $\displaystyle{\int_{\R^N}}\widehat k_\eps v\ge0$ for all $v\in L^\infty({\R^N})$ such that $v\ge0$ then, for such $v$, we also have $\langle\lambda,v\rangle\ge0$, which means that $\lambda\ge0$.

For $v\in L^\infty(\R^N)$ set $v^+=\max\{v,0\}$, $v^-= (-v)^+$. Since $\widehat k_\eps(|D^\sigma u^\eps|^2-g^2)\ge0$ then
\begin{align*}
\langle\lambda, &g^2 \, v^\pm\rangle\le \varliminf_{\eps\rightarrow0}\int_{{\R^N}}\widehat k_\eps|D^\sigma u^\eps|^2v^\pm\\
&= \varliminf_{\eps\rightarrow0}\left(\int_{{\R^N} }\widehat k_\eps|D^\sigma (u^\eps-u)|^2v^\pm-2\int_{{\R^N} }\widehat k_\eps D^\sigma (u^\eps-u)\cdot D^\sigma uv^\pm+\int_{{\R^N}}\widehat k_\eps|D^\sigma u|^2v^\pm\right)\\
&=\langle\lambda, |D^\sigma u|^2\, v^\pm\rangle,\quad\text{using \eqref{strong}},
\end{align*}
concluding that 
$$\langle\lambda, (|D^\sigma u|^2-g^2)\, v^\pm\rangle\ge 0.$$

The fact that $\widehat k_\eps\ge0$ and $u\in\K^\sigma_g$ imply $\widehat k_\eps(|D^\sigma u|^2-g^2)v^\pm\le0$ and, therefore, integrating and letting $\eps\rightarrow0$,
$\langle\lambda, (D^\sigma u|^2-g^2)\, v^\pm\rangle\le0$, and so
$$\langle\lambda, (|D^\sigma u|^2-g^2)\, v\rangle=0.$$
Writting $v=\frac{w}{|D^\sigma u|+g}$, for any $w\in L^\infty(\Omega)$, we conclude \eqref{lm2}.
\end{proof}

\section{The quasi-variational inequality with $\sigma$-gradient constraint}
\label{4}

In this section we consider a map $G$ such that
\begin{equation}\label{mapG}
G:L^{2^*}(\Omega)
\rightarrow L^\infty_\nu(\R^N)
\end{equation}
is a continuous and bounded operator, where $2^*$ is the Sobolev exponent as in \eqref{hinfty} for $0<\sigma<1$.

We recall that, whenever necessary, we still denote by $u$ the extension of $u$ by zero outside of $\Omega$.

We set
\begin{equation}
\label{kapaGu}
\K_{G[u]}^\sigma=\big\{v\in H^\sigma_0(\Omega):|D^\sigma v|\le G[u]\text{ a.e. in }\R^N \big\}
\end{equation}
and we shall consider the quasi-variational inequality
\begin{equation}
\label{iqv}u\in\K_{G[u]}^\sigma:\qquad\int_{\R^N} AD^\sigma u\cdot D^\sigma(v-u)\ge\int_{\Omega}f(v-u),\qquad\forall v\in \K_{G[u]}^\sigma.
\end{equation}

Generalising a compactness argument of \cite{KunzeRodrigues2000} where quasi-variational inequalities
 of this type were considered for the gradient case $\sigma=1$, we may give a general existence theorem.
 
 \begin{thm}\label{thm41}
 	Under the assumptions \eqref{aa}, for continuous and bounded operators $G$ satisfying  \eqref{mapG} and for any $f\in L^{2^\#}(\Omega)$, with $2^\#$ as in \eqref{hinfty}, there exists at least one solution for the quasi-variational inequality \eqref{iqv}.
 \end{thm}
\begin{proof}
Let $u=S(f,g)$ be the unique solution of the variational inequality \eqref{iv} with $g=G[w]$ for any $w\in L^{2^*}(\Omega)$. If $C_*>0$ denotes the Sobolev constant as in Theorem \ref{thm21}, since $f_2=0$ corresponds always to the solution $u_2=0$, we have the a priori estimate
\begin{equation}\label{cf}
\|u\|_{L^{2^*}(\Omega)}\le C_*\|u\|_{H^\sigma_0(\Omega)}\le \tfrac{C_*}{a_*}\|f\|_{L^{2^\#}(\Omega)}\equiv c_f,
\end{equation}
independently of $g\in L^\infty_\nu(\R^N)$.

Set $B_{c_f}=\big\{v\in L^{2^*}(\Omega):\|v\|_{L^{2^*}(\Omega)}\le c_f\big\}$ and define the nonlinear map $T=S\circ G:L^{2^*}(\Omega)\ni w\mapsto u\in L^{2^*}(\Omega)$ where $u=S(f,G[w])\in\K_{G[w]}^\sigma\cap\C^{0,\beta}(\overline{\Omega})$, $0<\beta<\sigma$ by \eqref{kapai}.

Clearly, \eqref{cf} implies $T(B_{c_f})\subset B_{c_f}$ and, by the continuity of $G$ and Theorem \ref{thm22}, $T$ is also a continuous map. On the other hand, $G$ is bounded, i.e. transforms bounded sets in $L^{2^*}(\Omega)$ into bounded sets of $L^\infty_\nu(\R^N)$ and $S\circ T$ is also a bounded operator. Therefore,  by \eqref{kapai}, $T(B_{c_f})$ is also a bounded set of $C^{0,\beta}(\overline{\Omega})$. Since the embedding $C^{0,\beta}(\overline{\Omega})\hookrightarrow L^{2^*}(\Omega)$ is compact, the Schauder fixed point theorem guarantees the existence of $u=Tu$, which solves \eqref{iqv}.
\end{proof}

\begin{example}\label{exx41}
Consider the operator $G:L^{2^*}(\Omega)\rightarrow L^\infty_\nu(\R^N)$ defined as follows:
\begin{equation}\label{by}
G[u](x)=F(x,w(x)),
\end{equation}
where $F:\R^N\times\R\rightarrow\R$ is a function bounded in $x\in\R^N$ and continuous in $w\in\R$, uniformly in $x\in\R^N$, satisfying, for some $\nu>0$,
\begin{equation}
\label{ex41b}
0<\nu\le F(x,w)\le\varphi(|w|)\qquad\text{ a.e. }x\in\R^N,
\end{equation}
and for some monotone increasing function $\varphi$. We may choose
\begin{equation}
\label{ex41c}
w(x)=\int_{\Omega}\vartheta(x,y)u(y)\,dy,
\end{equation}
where we give $\vartheta\in L^\infty\big(\R^N_x;L^{2^\#}(\Omega_y)\big)$. For $u_n\underset{n}{\rightarrow}u$ in $L^{2^*}(\Omega)$, from the estimate
\begin{multline*}
\sup_{x\in\R^N}\left|w_n(x)-w(x)\right|=\sup_{x\in\R^N}\left|\int_{\Omega}\vartheta(x,y)(u_n(y)-u(y))dy\right|\\
\le \sup_{x\in\R^N}\|\vartheta(x,\cdot)\|_{L^2{^\#}(\Omega)}\|u_n-u\|_{L^{2^*}(\Omega)}
\end{multline*}
and by the uniform continuity of $F$, we have
$$\|G[u_n]-G[u]\|_{L^\infty(\R^N)}=\|F(w_n)- F(w)\|_{L^\infty(\R^N)}\underset{n}{\rightarrow}0,$$
implying the continuity of $G$.

The boundedness of $G$ is a consequence of \eqref{ex41b} and therefore $G$ satisfies the assumptions of Theorem \ref{thm41}.
\end{example}

\begin{example}
Consider now the operator $G:H^\sigma_0(\Omega)	\rightarrow L^\infty_\nu(\R^N)$ given also by \eqref{by} with $F$ under the same assumptions as in the previous example, but now with
\begin{equation}
w(x)=\Phi(u)(x)=\int_{\R^N}\Theta(x,y)\cdot D^\sigma u(y)dy,
\end{equation}
where $\Theta\in\C^{0}\big(\overline \Omega_x;\R^N_y\big)$. Now $G$ is not only bounded but also completely continuous, since $\Phi:H^\sigma_0(\Omega)\rightarrow\C^0(\overline\Omega)$ is also completely continuous. Indeed, if $u_n\underset{n}{\lraup}u$ in $H^\sigma_0(\Omega)$-weak, then $w_n=\Phi(u_n)\underset{n}{\rightarrow}\Phi(u)=w$ in $\C^{0}(\overline{\Omega})$, because $\{D^\sigma u_n\}_n$, being bounded in $L^2(\R^N)^N$ implies $\{w_n\}_n$ uniformly bounded in $\C^0(\overline \Omega)$, 
$$|w_n(x)|\le\|\Theta(x,\cdot)\|_{L^2(\R^N)^N}\|D^\sigma u_n\|_{L^2(\R^N)^N},\qquad\forall x\in\overline{\Omega}$$
and also equicontinuous in $\overline{\Omega}$ by
$$|w_n(x)-w_n(z)|\le C\|\Theta(x,\cdot)-\Theta(z,\cdot)\|_{L^2(\R^N)^N}.$$
\end{example}

But $G$ is not defined in the whole $L^{2^*}(\Omega)$ and therefore we cannot apply Theorem~\ref{thm41} to solve \eqref{iqv}. Nevertheless, the solvability of \eqref{iqv} in this example is an immediate consequence of the following theorem.

\begin{thm}
Assume \eqref{aa} and let $f\in L^{2^\#}(\Omega)$ as previously. If the nonlinear and nonlocal operator $G$ satisfies 
\begin{equation}
\label{410}
G:H^\sigma_0(\Omega)\rightarrow L^\infty_\nu(\R^N)\ \text{ is bounded and completely continuous}
\end{equation}
then there exists a solution $u$ to the quasi-variational inequality \eqref{iqv}.
\end{thm}
\begin{proof}
Due to the estimate \eqref{cf} and the assumption \eqref{410}, the proof is analogous by applying the Schauder fixed point theorem to the nonlinear completely continuous map $$T=S\circ G: H^\sigma_0(\Omega)\ni w\mapsto u=S(f,G[w])\in H^\sigma_0(\Omega).$$
\end{proof}

\begin{example}
By restricting the domain of $G$ and using the same type of Carath\'eodory function $F$ as in Example \ref{exx41}, we can introduce the superposition operator
\begin{equation}\label{super}
G[u](x)=F(x,u(x)),\qquad u\in\C^0(\overline{\Omega}),\ x\in\R^N.
\end{equation}

In order to guarantee that $G:\C^{0}(\overline{\Omega})\rightarrow L^\infty_\nu(\R^N)$ is a continuous and bounded operator in an appropriate space to obtain a fixed point, we need to require that the function $F:\R^N\times\R\rightarrow\R$ is a bounded function in $x\in\R^N$ in each compact for the variable $u$, continuous in $u\in\R$ uniformly in $x\in\R^N$, and satisfying \eqref{ex41b}, where the function $\varphi$ is continuous and satisfies only
\begin{equation}\label{412}
0<\nu\le \varphi(t),\qquad t\in\R.
\end{equation}

This situation is covered by the next theorem.
\end{example}

\begin{thm}
Assume \eqref{aa}, let $f\in L^{1}(\Omega)$ and the functional $G$ be such that
\begin{equation}
\label{413}G:\C^0(\overline{\Omega})\rightarrow L^\infty_\nu(\R^N)\qquad\text{ is a continuous operator}.
\end{equation}
Then there exists a solution of the quasi-variational inequality \eqref{iqv}.
\end{thm}
\begin{proof}
As before, we set $T=S\circ G:\C^0(\overline{\Omega})\rightarrow H^\sigma_0(\Omega)$ and for  $w\in\C^0(\overline{\Omega})$, $u=S(f,G[w])$ solves \eqref{iv} with $g=G[w]$.

In order to apply the Leray-Schauder principle, we set
$$\mathscr{S}=\big\{w\in\C^0(\overline{\Omega}):w=\theta Tw,\,\theta\in[0,1]\big\}$$
and we show that $\mathscr{S}$ is a priori bounded in $L^\infty(\Omega)$. For any $w\in\mathscr{S}$, $u=Tw$ solves \eqref{iv} with $g=G[w]$. Hence we have, noting that $w=\theta u$,
\begin{equation*}
\|w\|_{L^\infty(\Omega)}\le \kappa \|D^\sigma w\|_{L^2(\R^N)^N}
\le \kappa\theta\|D^\sigma u\|_{L^2(\R^N)^N}\le\tfrac{\kappa^2}{a^*}\|f\|_{L^1(\Omega)}
\end{equation*}
if $\kappa>0$ is the constant of \eqref{maj}, by Theorem \ref{thm21}, and this a priori estimate is independent of $G$.

Since, by \eqref{inclusions}, $T(\C^0(\overline{\Omega}))\hookrightarrow \C^{0,\beta}(\overline{\Omega})\hookrightarrow \C^0(\overline{\Omega})$ and this last embedding is compact, we may conclude that $T$ is a completely continuous mapping into a closed ball of $\C^0(\overline{\Omega})$ and its fixed point $u=Tu$ solves \eqref{iqv}.
\end{proof}

It is clear that in general we cannot expect the uniqueness of solution to quasi-variational inequalities of the type \eqref{iqv}. However, the Lipschitz continuity of the solution map $f\mapsto u$ to the variational inequality \eqref{iv}, given by Theorem \ref{thm21}, allows us to obtain, via the strict contraction Banach fixed point principle, a uniqueness result in a special case of ``small'' and controlled variations of the convex sets for the quasi-variational situation with separation of variables in the nonlocal constraint $G$.

We denote, for $R>0$,
$$B_R=\big\{v\in H^\sigma_0(\Omega):\|v\|_{H^\sigma_0(\Omega)}\le R\big\}.$$
\begin{thm}
Let $f\in L^{2^\#}(\Omega)$, $\varphi\in L^\infty_\nu(\R^N)$ and
\begin{equation}
\label{415}
G[u](x)=\varphi(x)\Gamma(u),\qquad x\in\R^N,
\end{equation}
where $\Gamma:H^\sigma_0(\Omega)\rightarrow\R^+$  is a functional satisfying

{\em i)} $0<\eta(R)\le\Gamma(u)\le E(R),\quad\forall u\in B_R$,

{\em ii)} $|\Gamma(u_1)-\Gamma(u_2)|\le\gamma(R)\|u_1-u_2\|_{H^\sigma_0(\Omega)},\quad\forall u_1,u_2\in B_R,$

\noindent for sufficiently large $R\in\R^+$, with $\eta, E$ and $\gamma$ being monotone increasing positive functions of $R$. 

Then the quasi-variational inequality \eqref{iqv} has a unique solution, provided
\begin{equation}
2C_\#\,\frac{\gamma(R_f)}{\eta(R_f)}\,\|f\|_{L^{2^\#}(\Omega)}<1,
\end{equation}
where $R_f\equiv C_\#\|f\|_{L^{2^\#}(\Omega)}$ with $C_\#=C_*/a_*$ and $C_*$ is the constant of the Sobolev embedding as in \eqref{cf}.
\end{thm}
\begin{proof}
	Let $S:B_R\ni v\mapsto u\in H^\sigma_0(\Omega)$ be the solution map with $u=S(f,G[v])$ being the unique solution of the variational inequality \eqref{iv} with $g=G[v]$.
	
	The a priori estimate \eqref{cf} implies $S(B_{R_{f}})\subset B_{R_{f}}$.
	
	Given $v_i\in B_R$, let $u_i=S(v_i)=S(f,\varphi\,\Gamma(v_i))$, $i=1,2$, and choose $\mu=\frac{\Gamma(v_2)}{\Gamma(v_1)}>1$, without loss of generality.
	
	Setting $g=\varphi\,\Gamma(v_1)$, we have $\mu\,g=\varphi\,\Gamma(v_2)$ and
	$$S(\mu\,f,\mu\,g)=\mu S(f,g),$$
	$$\mu-1=\frac{\Gamma(v_2)-\Gamma(v_1)}{\Gamma(v_1)}\le\frac{\gamma(R_f)}{\eta(R_f)}\|v_1-v_2\|_\sigma$$
	by recalling the assumptions i) and ii) and denoting $\|w\|_\sigma=\|w\|_{H^\sigma_0(\Omega)}$ for simplicity.
	
	Consequently, using \eqref{cf} and \eqref{lip} with $f_1=f$ and $f_2=\mu\,f$, we have
	\begin{align*}
	\|S(v_1)-S(v_2)\|_\sigma&\le \|S(f,g)-S(\mu f,\mu g)\|_\sigma+\|S(\mu f,\mu g)-S( f,\mu g)\|_\sigma\\&\le(\mu-1)\|u_1\|_\sigma+(\mu-1)C_\#\|f\|_{L^{2^\#}(\Omega)}\\
	&\le 2C_\#(\mu-1)\|f\|_{L^{2^\#}(\Omega)}\\
	&\le 2C_\#\frac{\gamma(R_f)}{\eta(R_f)}\|v_1-v_2\|_\sigma\|f\|_{L^{2^\#}(\Omega)}
	\end{align*}
	and the conclusion of the theorem follows immediately.
\end{proof}

\begin{example} We can take $\Gamma$ of the form
	$$\Gamma(u)=\int_{\R^N}e(y,u(y),D^\sigma u(y))\,dy,\quad u\in H^\sigma_0(\Omega),$$
	with $e:\R^N\times\R\times\R^N\rightarrow[\eta,\infty)$, for some $\eta>0$, under a local Lipschitz condition of the type
	$$|e(y,v,\xi)-e(y,w,\zeta)|\le\gamma(R)\big(|v-w|+|\xi-\eta|\big)$$
	for $|v|$, $|w|$, $|\xi|$ and $|\zeta|$ less or equal to $R$.

\end{example}

\begin{rem}
	Assumptions {\em i)} and {\em ii)} have been used in Appendiz B of \cite{HintermullerRautenberg2013}  under the implicit assumptions of smallness of the term $f$, and in \cite{RodriguesSantos2019} in a simplified and more precise form in the case of gradient type (i.e. $\sigma=1$) and for a class of general operators of p-Laplacian type.
	\end{rem}

\begin{rem} The existence of solution of the  quasi-variational inequality \eqref{iqv} is obtained in this section by finding a fixed point of the map $w\mapsto S(f,G[w])=u$, under suitable assumptions. But when $u=S(f,G[w])$ is the solution of \eqref{iv} then there exists $\lambda\in L^\infty(\R^N){\bs '}$ such that $(u,\lambda)$ solves problem \eqref{lm1}-\eqref{lm2} with data $(f,G[w])$. In particular, when $u$ is a fixed point  $u=S(f,G[u])$ it solves the quasi-variational inequality, and we immediately get existence of a solution $(\lambda,u)$ of problem \eqref{lm1}-\eqref{lm2} for the quasi-variational case.
\end{rem}

\section*{Acknowledgements}
	
	The research of J. F. Rodrigues was partially done under the framework of the project PTDC/MAT-PUR/28686/2017 at CMAFcIO/ULisboa and L. Santos was partially supported by the Centre of Mathematics the University of Minho through the Strategic Project PEst UID/MAT/00013/2013.

\def\ocirc#1{\ifmmode\setbox0=\hbox{$#1$}\dimen0=\ht0 \advance\dimen0
by1pt\rlap{\hbox to\wd0{\hss\raise\dimen0
\hbox{\hskip.2em$\scriptscriptstyle\circ$}\hss}}#1\else {\accent"17 #1}\fi}

\end{document}